\documentclass[11pt]{amsart}
\usepackage{amsbsy,amssymb,amscd,amsfonts,latexsym,amstext,delarray, amsmath,graphicx,color,caption, amsthm, enumerate, verbatim}
\usepackage{amstext,amsopn, mathtools, bbm}
\usepackage{graphicx}
\usepackage{mathtools}
\usepackage{hyperref}
\usepackage{enumitem}
\usepackage{epsfig}

\input xy

\newcommand\Z{{\mathbb Z}}
\newcommand\R{{\mathbb R}}

\newcommand\T{{\mathbb T}}

\hypersetup{
	colorlinks=true,
	linkcolor=black,
	filecolor=black,      
	urlcolor=black,
	citecolor = black
}

\numberwithin{equation}{section}

\begin{document}
	
	\newtheorem{example}{Example}[section]
	\newtheorem{lemma}{Lemma}[section]
	\newtheorem{thm}{Theorem}[section]
	\newtheorem{prop}[lemma]{Proposition}
	\newtheorem{cor}{Corollary}[section]
	
	\theoremstyle{remark}
	\newtheorem{remark}{Remark}[section]
	\newtheorem{definition}{Definition}[section]
	
\author{Alex Barron}

\title{An $L^4$ maximal estimate for quadratic Weyl sums}

\maketitle

\newcommand{\Addresses}{{
		\bigskip
		\footnotesize

		\textsc{Department of Mathematics, University of Illinois Urbana-Champaign,
			Urbana, IL 61801, USA}\par\nopagebreak
		\textit{E-mail address}: \texttt{aabarron@illinois.edu} }}

\begin{abstract}
	We show that $$\bigg\|\sup_{0 < t < 1} \big|\sum_{n=1}^{N} e^{2\pi i (n(\cdot) + n^2 t)}\big| \bigg\|_{L^{4}([0,1])} \leq C_{\epsilon} N^{3/4 + \epsilon}$$ and discuss some applications to the theory of large values of Weyl sums. This estimate is sharp for quadratic Weyl sums, up to the loss of $N^{\epsilon}$.
\end{abstract}

\section{Introduction} In this paper we prove the following maximal estimate for quadratic Weyl sums. \begin{thm}\label{prop:WeylMaximal}  For any $\epsilon > 0$ there exists $C_{\epsilon} > 0$ such that \begin{equation} \label{eq:mainMax} \bigg\|\sup_{0 < t < 1} \big|\sum_{n=1}^{N} e^{2\pi i (n(\cdot) + n^2 t)}\big| \bigg\|_{L^{4}([0,1])} \leq C_{\epsilon} N^{3/4 + \epsilon}.\end{equation}
\end{thm} \noindent Some applications of this result are discussed below in Section \ref{sec:application}. The dependence on $N$ is sharp up to the loss of $N^{\epsilon}$, as can be seen by observing that the supremum is essentially $N$ for each $x\in[0, 10^{-6} N^{-1}]$. A similar argument shows that we cannot take $p > 4$ in \eqref{eq:mainMax} without increasing the loss in $N$. Moreover, there exists a set $E \subset [0,1]$ with $|E| \gtrsim 1$ and $\sup_{t}|\sum_{n=1}^{N} e^{2\pi i (nx + n^2 t)}| \gtrsim N^{3/4}$ when $x\in E$ (see \cite{MV} or the end of Section \ref{sec:dimension} below), and so the loss in $N$ cannot be improved if $p$ is decreased (up to the factor of $N^{\epsilon}$). 

  In \cite{MV} Moyua and Vega proved that \begin{equation} \label{eq:L4MV} \bigg\|\sup_{0 < t < 1} \big|\sum_{n=1}^{N} a_n e^{2\pi i (n(\cdot) + n^2 t)}\big| \bigg\|_{L^{p}([0,1])} \leq C_{\epsilon} N^{1/3 + \epsilon} \big(\sum_{n=1}^{N} |a_n|^2 \big)^{1/2}, \ \ \ \ 1 \leq p \leq 6, \end{equation} and when $p = 6$ this is sharp up to the loss of $N^{\epsilon}$ (let $a_n = 1$ and restrict to $[0, 10^{-6}N^{-1}]$ as above). Theorem \ref{prop:WeylMaximal} improves \eqref{eq:L4MV} in the range $1 \leq p \leq 4$ in the special case where $a_n = 1$ for each $n$. It is conjectured that we have the stronger estimate \begin{equation} \label{eq:conjecture} \bigg\|\sup_{0 < t < 1} \big|\sum_{n=1}^{N} a_n e^{2\pi i (n(\cdot) + n^2 t)}\big| \bigg\|_{L^{p}([0,1])} \leq C_{\epsilon} N^{1/4 + \epsilon} \big( \sum_{n=1}^{N} |a_n|^2 \big)^{1/2}, \ \ \ \ 1 \leq p \leq 4.
 \end{equation} This conjecture is still open, though Theorem \ref{prop:WeylMaximal} implies that it holds in the Weyl sum case. Maximal estimates of the type \eqref{eq:mainMax} for more general Weyl sums were recently considered in \cite{CSmax} by Chen and Shparlinski, who studied maximal analogues of the Vinogradov mean-value estimates proved by Bourgain, Demeter, and Guth in \cite{BDG} using $\ell^2$ decoupling, and by Wooley in \cite{Wo1} using efficient congruencing.  
 
 Maximal estimates of the form \eqref{eq:L4MV}, \eqref{eq:conjecture} imply pointwise convergence results for the linear Schr\"{o}dinger equation on the torus $\T,$ which we identify with $[0,1]$. In particular, by using Littlewood-Paley theory and the Stein maximum principle we see that \eqref{eq:conjecture} with $p = 2$ is equivalent to $$\lim_{t\rightarrow 0} e^{it\Delta_{\T}}f(x) = f(x) \text{ a.e. }, \ \ \ \  \text{ for all } f\in H^{1/4 +}(\T)$$ 
 
 \noindent (see for example Appendix A in \cite{P}). The corresponding pointwise convergence result for the Schr\"{o}dinger operator on $\R$ is known to hold, and was originally proved by Carleson in \cite{Ca}. Indeed, one has $$\lim_{t\rightarrow 0} e^{it \Delta_{\R}}f(x) = f(x) \text{ a.e. }, \ \ \ \ f \in H^{1/4}(\R) $$ and this is sharp as shown in \cite{DK}. More generally, recent developments in Fourier restriction theory have led to the near-complete resolution of the pointwise convergence problem for the linear Schr\"{o}dinger equation on $\R^n$, originally posed in \cite{Ca} and \cite {DK}. The problem is to determine the minimal $s > 0$ for which \begin{equation} \label{eq:maximalest} \|\sup_{0 < t < 1} |e^{it\Delta_{\R^n}} f| \|_{L^{p} (B^n (0,1))} \leq C \|f\|_{H^s (\R^n)}\end{equation} 
 \noindent for some $p \geq 1$. In the breakthrough works of Du, Guth, and Li \cite{DGL} and Du and Zhang \cite{DZ} the estimate \eqref{eq:maximalest} was proved with sharp Sobolev index (up to the endpoint) in all dimensions $n \geq 2$ (the case $n = 2$ was proved with $p = 3$ in \cite{DGL}, and the case $n \geq 3$ was proved with $p = 2$ in \cite{DZ}). Their arguments involve several recent advancements in Fourier restriction theory, including the polynomial method, refined Strichartz estimates, $\ell^2$ decoupling, and a fractal refinement of the broad-narrow reduction of Bourgain and Guth \cite{BG}. 
 
 In the case $n = 1$ the argument due to Carleson is much simpler than the recent proofs given in higher dimensions. Unfortunately this argument does not readily transfer to the case of the torus, since the kernel of the periodic Schr\"{o}dinger operator is more erratically behaved than the kernel of the Euclidean operator. Indeed, in the periodic case the kernel is given by the quadratic Weyl sum $$w_N(x,t) = \sum_{n=1}^{N} e^{2\pi i (nx + n^2 t)}$$ which is known to be large whenever $(x,t)$ is close to a rational point $(b/q, a/q)$ with $1 \leq q < N^{1/2}$. Our main theorem confirms that \eqref{eq:conjecture} holds for this kernel, though the methods do not appear to be enough to prove the conjecture for general solutions.   
 
 The proof of Theorem \ref{prop:WeylMaximal} is inspired by the argument of Du and Zhang in \cite{DZ}. Three important features of their work play a key role in our argument: induction on the scale $N$, a version of parabolic rescaling, and the `one-dimensional' structure of the set where the supremum is attained. However, there are new complications of a number-theoretic nature that arise. 
 
 In order to prove an estimate like \eqref{eq:mainMax} one hopes to show that the supremum in $t$ is `on average' no worse than $N^{3/4 + \epsilon}$. More specifically, we already lose $N^{3/4}$ from $x \in [0, cN^{-1}]$ and so the goal is to show that remaining points do not contribute much more to the $L^4$ norm. From classical Weyl estimates we know that the supremum is larger than $N^{3/4}$ whenever $|x-a/q| < 1/(100N)$ for rationals $a/q$ with $1 \leq q < N^{1/2}$ and $q$ odd (some conditions on $a$ are required if $q$ is even; see for example \cite{O}). These intervals $I_{a/q}$, which correspond to `major arcs' in the Hardy-Littlewood circle method, are pairwise-disjoint for $q$ in this range and can be shown to cover a subset  $E\subset \T$ of measure $|E|\gtrsim 1$ (see for example \cite{MV} or the argument at the end of Section \ref{sec:dimension} below). One must then show that the Weyl sum cannot frequently be larger than $N^{3/4}$ on $E$, and also that the supremum cannot frequently be larger than $N^{3/4}$ if $x$ is outside of $E$. The latter is more challenging, since classical Weyl estimates do not give good information about lower bounds for the sum if $x$ is outside the union of the $I_{a/q}$. Moreover, the times $t$ at which the supremum is attained will be `major arc' points, and so it is not clear if it is possible to exploit the cancellation one typically obtains from the `minor arcs' in the circle method.

 To get around these issues we consider a certain collection of axis-parallel rectangles $Q = I\times J$ in $\T^2$ where the Weyl sum attains large values,  following the set-up of the argument in \cite{DZ}. The rectangles are then grouped according to the Diophantine approximation properties of $t \in J$. If the times $t\in J$ are sufficiently close to a rational $a/q$ with $q \geq N^{c\epsilon^2}$ we show that we can control their contribution by induction on the scale $N$. On the other hand, if $t \in J$ are close to a rational $a/q$ with small denominator $q$ we will see we can directly apply a local-in-time maximal estimate of Moyua and Vega (see Theorem \ref{thm:MV} below). 
 
 There are some similarities between this approach and the broad-narrow reduction of Bourgain and Guth \cite{BG} which is now a common tool in Fourier restriction theory and in particular played a key role in \cite{DGL} and \cite{DZ}. However, these connections appear to be only heuristic, and we have not found a more precise way to bring into play tools from restriction theory. It would be of interest, for example, to see if it is possible to approach the conjectured bound \eqref{eq:conjecture} using some kind of refined $\ell^2$ decoupling estimates (for example similar to those proved in \cite{GIOW}). 
 
 \subsection{An application of Theorem \ref{prop:WeylMaximal}: Large values of Weyl sums} \label{sec:application} In Section \ref{sec:cor} we will derive a few direct corollaries of our maximal estimate. For example, let $\mathcal{S}_\alpha (N)$ denote the set of $x \in \T$ such that $$\sup_{0< t < 1}\bigg| \sum_{n=1}^{N} e^{2\pi i (xn + tn^2)} \bigg| \geq N^{\alpha}.$$ We will show in Corollary \ref{cor:size} that if $3/4 < \alpha \leq 1$ then \begin{equation} \label{eq:setBound} |\mathcal{S}_{\alpha} (N)|  \lesssim_{\epsilon} N^{3 - 4\alpha + \epsilon}.\end{equation} This result gives us information about pointwise bounds for the Weyl sum assuming only information about $x$, and should be compared with the classical estimate in Proposition \ref{prop:WeylLower} below. The estimate \eqref{eq:setBound} should also be compared with recent work of Chen and Shparlinski, who obtained similar bounds for exponential sums in \cite{CS}. In the Weyl sum case their estimates in \cite{CS}, Lemma 3.2 and Corollary 3.8, show that $$|S_{\alpha}(N)| \lesssim_{\epsilon}N ^{6 - 7\alpha + \epsilon},$$ which our result improves. However, the estimates in \cite{CS} extend to the case of weighted sums and higher order phases, which ours does not. 
 
 In Section \ref{sec:dimension} we consider an analogue of a related problem raised by Chen and Shparlinski. In \cite{CS2} and \cite{CS3} the authors analyze the following set which characterizes points contributing to large values of $w_N$ as $N\rightarrow \infty$:  $$\mathcal{E}_{\alpha} = \{(x,t) \in \T^2 : |w_{N}(x,t)| \geq N^{\alpha} \text{ for infinitely many } N \in \mathbb{N}\}.$$ It is not too hard to show that $|\mathcal{E}_{\alpha}| = 0$ if $\alpha > 1/2,$ since on `most' of $\T^2$ we have square-root cancellation with $|w_N| \lesssim N^{1/2}$ (see for example \cite{CS2} for a proof). On the other hand, understanding the Hausdorff dimension of $\mathcal{E}_{\alpha}$ is a much more subtle problem.  
 
 Given a Borel set $\mathcal{S}$ we will write $\dim (\mathcal{S})$ to denote its Hausdorff dimension, so that $$\dim (\mathcal{S}) = \inf \{ s > 0 : \mathcal{H}^s (\mathcal{S}) = 0 \}, \ \ \ \ \ \mathcal{H}^s = s\text{-dimensional Hausdorff measure}. $$ In \cite{CS2} and \cite{CS3} Chen and Shparlinski proved that \begin{equation}\label{eq:CSbound} \min(3/2, 3(1-\alpha) ) \leq  \text{dim}(\mathcal{E}_{\alpha}) \leq \frac{8(1-\alpha)}{2-\alpha}.\end{equation} It was also recently shown by Chen, Kerr, Maynard, and Shparlinski in \cite{CKMS} that there exist constants $a,b > 0$ such that the set $$\{ (x,t)\in \T^2 : aN^{1/2} \leq |w_{N}(x,t)| \leq b N^{1/2} \text{ for infinitely many } N \in \mathbb{N} \} $$ has full Lebesgue measure (and in fact this holds for more general Weyl sums associated to the moment curve). 
 
 We will consider a one-dimensional version of these types of estimates. Define $$\mathcal{O}_{\alpha} = \{x \in \T : \sup_{0 < t < 1} |w_{N}(x,t)| \geq N^{\alpha} \text{ for infinitely many } N\in \mathbb{N} \} $$ As a consequence of Theorem \ref{prop:WeylMaximal} we will obtain the following. 
 
 \begin{prop}\label{prop:dim} If $3/4 \leq \alpha \leq 1$ then $$ \dim(\mathcal{O}_{\alpha})  = 4(1-\alpha).$$  
 \end{prop}  \noindent We prove this result in Section \ref{sec:dimension}. This proposition can be seen as a `Weyl sum analogue' of the divergence set estimates for the Euclidean Schr\"{o}dinger equation proved in \cite{BBCR}, \cite{DZ}, \cite{LR}. Indeed the method of proof is similar. 

We remark that the lower bound in the proof of Proposition \ref{prop:dim} does not require our maximal estimate. Instead it follows from the classical observation that $\mathcal{O}_{\alpha}$ contains points which are well-approximated by rationals with odd denominators. A theorem of Jarn\'{\i}k gives us good information about the dimension of such sets. The set $\mathcal{O}_{\alpha}$ could contain many other points, but Theorem \ref{prop:WeylMaximal} will be used to show that these additional points do not increase the Hausdorff dimension. 

\subsection{Connection with recent estimates of Weyl sums}\label{sec:otherWeyl}  We briefly discuss some other recent estimates of Weyl sums that are related to Theorem \ref{prop:WeylMaximal} and the corollaries mentioned above. There is a wide body of research on Weyl sums, and we do not claim to give anything close to a complete summary here. 

Motivated by problems related to the Talbot effect from optics, the authors in \cite{BPPSV} have recently proved the following essentially sharp bound: for any $\epsilon >0$ and $\gamma \in [0,1)$ one has \begin{equation} \label{eq:talbot} \sup_{x \in [0,1]}\big| \sum_{n=1}^{N} e^{2\pi i (xn + (x+\gamma)n^2)}  \big| \lesssim_{\epsilon, \gamma} N^{3/4 + \epsilon}.\end{equation}  Erdoğan and Shakan had previously shown in \cite{ES} that \begin{equation} \label{eq:talbot2} \sup_{x \in [0,1]}\big| \sum_{n=1}^{N} e^{2\pi i (xn + (a-rx)n^2)}  \big| \lesssim_{\epsilon, a,r} N^{4/5 + \epsilon}. \end{equation} 
\noindent In \cite{CStalb} Chen and Shparlinski also prove analogues of these types of estimates involving supremums over preimages of Lipschitz functions. 

If more information is known about the coefficients one can obtain improvements. For example, Oskolkov showed that if $r = m/2\pi$ and $a = k/2$  for $m,k$ odd integers then the bound in \eqref{eq:talbot2} can be improved to $N^{1/2}$ (see \cite{Os}, Lemma 2). For more information about the motivation for these estimates and their relation to the Talbot effect see for example \cite{ET}, Section 2.3, or the introductions of \cite{ES} and \cite{Os}. 

Note that the dependence on $N$ in \eqref{eq:talbot} is the same as in our $L^4$ maximal estimate. There are some overlapping ideas in parts of the proof of Theorem \ref{prop:WeylMaximal} and the proof of \eqref{eq:talbot}, in particular related to the role of Diophantine approximation of the coefficient of $n^2$. However, neither of these results appears to imply the other. Indeed, to prove Theorem \ref{prop:WeylMaximal} one must understand the size of the region where the supremum in $t$ is much larger than $N^{3/4}$. Also note that the constant in \eqref{eq:talbot} depends on $\gamma$, which leads to additional obstacles (for example if we shift the frequencies and try to exploit Galilean invariance). On the other hand, Theorem \ref{prop:WeylMaximal} implies that if $t(x)$ is any measurable function of $x$ (and in particular a linear function of $x$) then $$ \big(\int_{\T} \big| \sum_{n=1}^{N} e^{2\pi i (xn + t(x)n^2)} \big|^4 dx \big)^{1/4} \leq C_{\epsilon}N^{3/4 + \epsilon},$$ but this is not enough to conclude \eqref{eq:talbot}. 

The inequalities \eqref{eq:talbot}, \eqref{eq:talbot2} imply dimension bounds on the graphs of certain solutions to the linear Schr\"{o}dinger equation if they are restricted to lines, as studied in \cite{ES}, \cite{Os}, and \cite{OC}. They are also useful for proving similar estimates for the Hilbert-transform-like function $$H(x,t) = \sum_{\substack{ n\in \Z \\ n \neq 0 }}  \frac{1}{n} e^{2\pi i (xn + tn^2)},$$ which is studied in depth by  Chakhkiev and Oskolkov in \cite{OC}. Indeed, by using Littlewood-Paley theory and summation by parts one can use estimates for the Weyl sum to obtain information about $H$; estimates for $H$ then imply bounds on the dimension of graphs of solutions to the linear Schr\"{o}dinger equation with initial data of bounded variation  (see for example \cite{ES}, Theorem 2.2.). In \cite{OC} the authors consider $H$ restricted to slices with $x$ or $t$ fixed, and characterize the H\"{o}lder continuity of these slices in terms of the Diophantine approximation properties of $x$ or $t$. In particular, in \cite{OC}, Corollary 4, the authors calculate the Hausdorff dimension of the set of $x$ for which the restriction $H|_{x}$ is $\alpha$-H\"{o}lder continuous, and prove analogous bounds for $H|_{t}$. The methods of proof bear a resemblance to the proof of the lower bound in Proposition \ref{prop:dim}, and in particular the Jarn\'{\i}k theorem (Theorem \ref{thm:Jarnik} below) plays an important role in their argument.

\subsection{Notation} We will write $A \lesssim B$ to mean there is a uniform constant $C > 0$ such that $A \leq C B$. If there is a constant $C(\beta)$ that depends on some parameter $\beta$ such that $A \leq C(\beta)B$ we will write $A \lesssim_{\beta} B$. We similarly define $A \gtrsim B$ and $A\gtrsim_{\beta}B,$ and write $A \sim B$ if $A \lesssim B$ and $A \gtrsim B$. Throughout the paper we occasionally let $C, c, c',$ etc. denote uniform constants which may change from line-to-line. We also let $C_{\epsilon}$ be the distinguished constant appearing in Theorem \ref{prop:WeylMaximal}. Once we fix $\epsilon > 0$ the constant $C_{\epsilon}$ does not change line-to-line.

 \subsection*{Acknowledgments} The author thanks the anonymous referees for providing helpful comments which have improved the presentation of the paper. The author also thanks M. B. Erdoğan for suggesting several references and for helpful discussions related to the material in this paper.

\section{The proof of Theorem \ref{prop:WeylMaximal}}\label{sec:Weyl}

Certain collections of axis-parallel rectangles will play a key role in the proof of Theorem \ref{prop:WeylMaximal}. Fix $3/4 \leq \alpha \leq 1$ and a small $\eta > 0$, and tile $[0,1]^2$ by axis-parallel rectangles of dimensions approximately $N^{-2 + \alpha - \eta} \times N^{-4 + 2\alpha - 2\eta}$. Let $\Lambda_{\alpha}$ denote the rectangles in this tiling. We will work with certain sub-collections of $\Lambda_{\alpha}$. \begin{definition} Let $\pi_x: [0,1]^2 \rightarrow [0,1]$ denote projection in the $x$ variable. Let $\mathcal{Q}$ be a sub-collection of $\Lambda_{\alpha}$. We say that $\mathcal{Q}$ is \textit{one-dimensional} at scale $(N, \alpha)$ if for each $x_0 \in [0,1]$ there are at most two rectangles $Q \in \mathcal{Q}$ such that $x_0 \in \pi_{x}(Q)$  \end{definition} The choice of scale is motivated by the following `locally-constant' behavior of exponential sums.

\begin{lemma}\label{lem:const} Suppose $0 < \alpha \leq 1$ and fix a small $\eta > 0$. Suppose $c N^{\alpha} \leq |w_{N}(x_0,t_0)| \leq C N^{\alpha}$ and let $Q_{\alpha}$ be an axis-parallel rectangle of dimensions $N^{-2 + \alpha -\eta} \times N^{-3 + \alpha -\eta}$ that contains $(x_0,t_0)$. Then $N^{\alpha } \lesssim |w_{N}(x,t)| \lesssim  N^{\alpha }$ for all $(x,t) \in Q_{\alpha}$.\end{lemma} 
\begin{proof} This is a direct consequence of the mean value theorem (see also \cite{Wo} or \cite{CS} for more general versions of this estimate).\end{proof} \noindent Below we suppress the role of $\eta$ from the notation, since it can be chosen to be much smaller than all other small parameters under consideration (e.g. take $\eta = \epsilon^{100}$).

We now turn to the proof of Theorem \ref{prop:WeylMaximal}. To slightly simplify the notation we will write $$u(x,t) = w_N(x,t).$$ Fix $\epsilon > 0$  for the rest of the argument. If $\sup_{t}|u(x,t)| \leq N^{3/4}$ then the estimate is trivial, so it suffices to estimate the integral on the set $E$ where the supremum is larger than $N^{3/4}$. Partition $E$ into a disjoint union of sets $E_k$ where the supremum is in $[2^k, 2^{k+1}),$ with $N^{3/4} \leq 2^k \leq N.$ The number of $k$ is $O(\log N)$ and so it suffices to estimate
\begin{equation}\label{eq:maxFixedScale} \big(\int_{E_{k}} \sup_{0 < t < 1}|u(x,t)|^4 dx \big)^{1/4} 
\end{equation}
	
\noindent Our argument will involve induction on the scale $N$. In particular note that since Theorem \ref{prop:WeylMaximal} is trivial if $N \lesssim_{\epsilon} 1$ we may assume by induction that it holds for any scale $N'$ with $N' < N/2$.

Choose $\alpha \in [3/4, 1]$ such that $2^{k} \sim N^{\alpha}$, and cover $E_{k}$ by a union of intervals $I_\alpha$ of length approximately $N^{-2 +\alpha}$ which overlap only at their boundaries. For each $I_{\alpha}$ pick a point $x_{\alpha} \in I_{\alpha} \cap E_{k}$ and choose an interval $J_{\alpha}$ of length approximately $N^{-4 + 2\alpha}$ such that $\sup_{t}|u(x_{\alpha}, t)|$ is attained at some $t_{\alpha}\in J_{\alpha}$. Let $$Q_{\alpha} = I_{\alpha} \times J_{\alpha}.$$ Then the $Q_{\alpha}$ form a one-dimensional collection at scale $(N,\alpha)$. Moreover, by Lemma \ref{lem:const} there exist constants $c, C > 0 $ such that $$cN^{\alpha} \leq  |u(x,t)| \leq C N^{\alpha}, \ \ \ (x,t) \in Q_{\alpha},$$ and so \begin{equation}\label{eq:oneDim} \big(\int_{E_{k}} \sup_{0 < t < 1}|u(x,t)|^4 dx \big)^{1/4}   N^{1 -\alpha/2}\big(\sum_{Q_{\alpha}}  \int_{Q_{\alpha}} |u(x,t)|^4 dxdt \big)^{1/4}. \end{equation}

 Let $X$ denote the one-dimensional collection of rectangles $Q_{\alpha}$ appearing in $\eqref{eq:oneDim}$. 
 
 \subsection{Partitioning the rectangles in $X$} We will now partition the rectangles in $X$ into different groups $X_q$ according to the rational approximation properties of $t$ with $(x,t) \in Q_{\alpha}$. To clear up notation we now drop the subscript $\alpha$ from the $Q_{\alpha} \in X$.
 
 We will use the following dispersive estimate for $u$ due to Bourgain, which is a refinement of the classical Weyl bound.

\begin{lemma}[\cite{B}]\label{lem:Bourgain} Suppose there are integers $1 \leq q \leq N$ and  $0 \leq a < q$ with $(a,q) = 1$ such that $|t - a/q| \leq 1/qN.$ Then for all $x$ one has $$|u(x,t)|\lesssim \log N \frac{N}{q^{1/2}(1 + N|t-a/q|^{1/2} )}.$$
\end{lemma} 

\begin{proof} This is essentially proved in Lemma 3.18 in \cite{B}, although the details are only provided for sums of the form $$ W(x,t) = \sum_{|n| \lesssim N} \phi(n/N) e^{2\pi i (nx + n^2 t)},$$ where $\phi$ is a smooth cut-off equal to 1 on $[1,N]$ (for example). But we can write $u = W(\cdot, t) \ast D_{N}(x),$ where $D_{N}$ is the Dirichlet kernel. Therefore $$|u(x,t)| \leq \|W(\cdot, t)\|_{L^{\infty}}\|D_N\|_{L^1} \lesssim \log(N) \|W(\cdot, t)\|_{L^{\infty}},$$ using the well-known bound $\|D_N\|_{L^1} \lesssim \log(N)$.  The claim now follows by applying Lemma 3.18 from \cite{B}. 
\end{proof}

We choose $m \in \mathbb{N}$ such that $2^{-m/2}N \sim N^{\alpha}$, so $2^{m}$ is the dyadic scale of $N^{2-2\alpha}$. This choice of scale is motivated by Lemma \ref{lem:Bourgain}. Note that $2^{m} \lesssim N^{1/2}$ since $\alpha \geq 3/4$. 

Let $Q_j = I_j \times J_j$ index $Q \in X$. Then by Dirichlet's approximation theorem, for each $t \in J_j$ there exist integers $0 \leq a < q \leq N$  with $|t - a/q| < 1/Nq$ and $(a,q) = 1$. Lemma \ref{lem:Bourgain} then implies that there must exist some $t_j \in J_j$ and some integers $0\leq a_j <  q_j \leq N$ with $(a_j, q_j) = 1 $ such that $$ 2^{-m/2}N \lesssim \|u\|_{L^{\infty}(Q_j)} \lesssim  \frac{N\log(N)}{q_j^{1/2}(1 + N|t_j -a_j/q_j|^{1/2} )}.$$ In particular \begin{equation} \label{eq:qBound} 1 \leq q_j \leq c\log(N)^2 2^{m} \end{equation} and \begin{equation}\label{eq:tRange} |t_j - a_j/q_j| \leq c\log(N)^2 \frac{2^m}{q_{j} N^{2}}.
\end{equation} Note that \eqref{eq:tRange} guarantees $t_j$ is closer to the rational $a_j/q_j$ than what Dirichlet's theorem implies. We collect this result as the following lemma.

\begin{lemma}\label{lem:tApprox} Suppose $Q$ is an axis-parallel rectangle of dimension $N^{-2 + \alpha}\times N^{-4 + 2\alpha}$ with $\|u\|_{L^\infty(Q)} \geq N^{\alpha}$, where $u$ is a scale-$N$ quadratic Weyl sum. Suppose $t_\ast$ is the time at which the supremum is attained, and $|t_{\ast} - a/q| \leq \frac{1}{qN}$ with $0 \leq a < q\leq N$ and $(a,q) =1$. Then if $2^{m} \sim N^{2- 2\alpha}$ we have \begin{equation} \label{eq:qBoundGeneral} 1 \leq q \leq c\log(N)^2 2^{m} \end{equation} and in fact \begin{equation}\label{eq:tRangeGeneral} |t_\ast - a/q| \leq c\log(N)^2 \frac{2^m}{q N^{2}}.\end{equation}   
\end{lemma}

We now group $X$ according to the value of $q_j$ from the range \eqref{eq:qBound}. We say $q_j \sim Q_j$ if $$|t_j - a_j/q_j| \leq c\log(N)^2 \frac{2^m}{q_{j}N^2}$$ for some $0 \leq a_j < q_j$ with $(a_j, q_j) = 1$ and $$ 1 \leq q_j \leq c\log(N)^2 2^m ,$$ where $c$ is the constant appearing above. Note that $a_j/q_j$ is uniquely determined. Indeed, if $a'/q' \neq a_j/q_j$ and both $q_j \sim Q_j$ and $q' \sim Q_j$ then we must have $$\frac{1}{q' q_j} \leq \big|\frac{a'}{q'} - \frac{a_j}{q_j}\big| \leq 2c\log(N)^2 \frac{2^m}{N^2} \big( \frac{1}{q'} + \frac{1}{q_j} \big) $$ and therefore  $N^2 \lesssim \log(N)^2 2^{2m}$; but this is false since $2^{2m} \lesssim N$. A similar argument shows that if $q' = q_j$ then $a' = a_j$ whenever $q_j \sim Q_j$.  

We define $$X_q = \{ Q_j \in X :  q \sim Q_j \}.$$ Then $$X =  \bigcup_{q=1}^{c\log(N)^2 2^{m} } X_q $$ and the union is disjoint.

\subsection{Contribution of $X_q$ with $1 \leq q \leq N^{\delta}$} \label{sec:smallq} We begin by considering the contribution to \eqref{eq:oneDim} from rectangles in $X_q$ with $1 \leq q \leq N^{\delta}$, where $\delta> 0$ is a small parameter with $\delta < \epsilon^2$. In this case we directly estimate the integral in \eqref{eq:oneDim} without using induction. Note that for each $Q_j$ under consideration we have $$|t_j - a/q| \leq  c\log(N)^2 \frac{2^m}{q N^2}$$ for some choice of $1 \leq q \leq N^\delta$ and $0 \leq a < q$ with $(a,q) = 1$. Let $\Phi(q)$ denote the number of positive integers $a$ less than $q$ such that $(a,q)= 1$ (so $\Phi$ is the Euler totient function). Then the number of possible $a/q$ in this case is \footnote{ In fact we have the more precise estimate $\sum_{q = 1}^{N^{\delta}} \Phi(q) = \frac{3}{\pi^2} N^{2\delta} + O(\delta N^{\delta} \log(N))$ but this will be of no use here. See \cite{A}, Theorem 3.7, for a proof of this identity. } \begin{equation}\label{eq:Euler} \sum_{q=1}^{N^{\delta}} \Phi(q) \lesssim N^{2\delta}.\end{equation} Therefore after a loss of $CN^{2\delta}$ we may assume that there is one fixed $a/q$ as above such that all remaining $Q_j$ satisfy $|t_j - a/q| \leq  c\log(N)^2 \frac{2^m}{qN^2}$. Let $Y$ denote the collection of the remaining $Q$. Then by \eqref{eq:qBound} we have $$\bigcup_{Q \in Y} Q \subset [0,1] \times [ a/q - c\log(N)^2 \frac{2^m}{q N^2}, \ a/q + c\log(N)^2 \frac{2^m}{q N^2} ] := \T \times I,$$ with $|I| \lesssim_{\delta} 2^{m}N^{-2 + \delta} \sim N^{-2\alpha + \delta}.$ Since $\delta < \epsilon^2$ all losses of $N^{O(\delta)}$ will be negligible. We will suppress the dependence below and write $A \lessapprox B$ to mean $A \leq C(\delta)N^{O(\delta)}B$.

Notice that since the collection $Y$ is one-dimensional and contained in $\T\times I$ we have \begin{equation} \label{eq:smallqUpper}N^{1-\alpha/2} \big(\sum_{Q \in Y} \int_{Q}|u|^4 \big)^{1/4}  \lesssim \big( \int_{\T} \sup_{t \in I}|u(x,t)|^4 dx \big)^{1/4}.\end{equation}
\noindent To conclude the argument we will appeal to the following local estimate of Moyua and Vega. 
\begin{thm}[\cite{MV}] \label{thm:MV} Let $u$ be any solution to the linear Schr\"{o}dinger equation on $\T$ (not necessarily a Weyl sum). Suppose $u(x,0) = f(x)$ and $\widehat{f}$ is supported in $[-N,N]$. If $\eta > 0 $ then for any $\epsilon > 0$ we have $$\| \sup_{0 < t < \eta} |u(x,t)|\|_{L^4 (\T)} \lesssim_{\epsilon} N^{1/2 + \epsilon} \max(N^{-1}, \eta)^{1/4} \|f\|_{L^{2} (\T)}.$$  
\end{thm} \noindent When $\eta \sim N^{-1}$ this implies the sharp bound for the maximal operator with a restricted supremum. Note that in the `Euclidean window' $0 < t < N^{-2}$ we expect matters to reduce to known estimates for the solution on $\R$; this result gives an extension to the interval $0 < t < N^{-1}$. The theorem applies to any time interval of length $\eta$ since the interval can be translated by modulating the Fourier coefficients of $f$, and this does not change $\|f\|_{L^2}$.

Applying Theorem \ref{thm:MV} yields \begin{equation} \label{eq:MV} \big( \int_{\T} \sup_{t \in I}|u(x,t)|^4 dx \big)^{1/4} \lessapprox N^{3/4} \end{equation} (recall $|I| \lessapprox N^{-2\alpha} < N^{-1}$). Then from \eqref{eq:smallqUpper} and \eqref{eq:MV} we conclude that $$N^{1-\alpha/2}\big(\sum_{Q\in Y} \int_{Q} |u(x,t)|^4 dx dt \big)^{1/4} \lessapprox N^{3/4}, $$ as desired.

\subsection{Contribution of $X_q$ with $q > N^{\delta}$}\label{sec:largeq} We now turn to the rectangles belonging to some $X_q$ with $q > N^{\delta}$. We will control the contribution of each $X_q$ by using induction on the scale $N$. 

Suppose $(x,t) \in Q_j$ and $Q_j \in X_q$. We break the range of summation in $u$ into congruence classes mod $q$ to write \begin{align}\label{eq:decomp} u(x,t) &= \sum_{r=1}^{q} \sum_{l=0}^{\lfloor N/q \rfloor} e^{2\pi i ( (lq + r)x + (lq + r)^2t )} + O(q)  \\ \nonumber &= \sum_{r=1}^{q} e^{2\pi i (r^2t + rx)} \sum_{l=0}^{\lfloor N/q \rfloor} e^{2\pi i (l^2 q^2 t + lqx  )}e^{2\pi i (2rlqt)} + O(q).\end{align} By construction there is some integer $a$ with $(a,q) = 1$ such that $$|t_j - a/q| \lesssim \log(N)^2 \frac{2^m}{qN^2}.$$ Since $|t-t_j| \leq 1/N^2$ and $q \lesssim \log(N)^2 2^{m} $ we also have $|t - a/q| \lesssim \log(N)^2 \frac{2^m}{q N^2}.$ The choice of $a$ depends on $Q_j$, but below we will use a pointwise estimate that is uniform over possible values of $a$.

Since $|t - a/q| \lesssim \log(N)^2 \frac{2^m}{qN^2}$ we have $$2rlq|t - a/q| \lesssim \log(N)^2 2^{m} N^{-1}  $$ and therefore  \begin{equation} \label{eq:error} e^{2\pi i (2rlqt)} = e^{2\pi i (2rlq[t - a/q])} = 1 + O(\log(N)^2 2^{m} N^{-1}) \end{equation} with the implicit constant independent of $a$ and $q$. It follows from \eqref{eq:qBound}, \eqref{eq:decomp}, and \eqref{eq:error} that we can write \begin{equation}\label{eq:mainDecomp}
u(x,t) = \sum_{r=1}^{q} e^{2\pi i (r^2t + rx)} \sum_{l=0}^{\lfloor N/q \rfloor} e^{2\pi i (l^2 q^2 t + lqx  )} + E_q(x,t), \ \ \ \ \ (x,t) \in Q \in X_q\end{equation} where $E_q$ is supported on $\bigcup_{Q \in X_q} Q$ and $$|E_q(x,t)|\mathbbm{1}_{Q}(x,t) \lesssim q\frac{N}{q} \log(N)^2 2^{m} N^{-1} \lesssim \log(N)^2 2^{m}, \ \ \ \ \ Q \in X_q$$   

\noindent (note that we have absorbed the error $O(q)$ from \eqref{eq:decomp} into the definition of $E_q$). Define the function $E(x,t) = \sum_{q} E_{q}(x,t)$. Since the cubes $Q \in X$ are pairwise disjoint we have $$|E(x,t)| \lesssim \log(N)^2 2^{m}, \ \ \ \ (x,t) \in \bigcup_{Q \in X} Q.$$ Also define functions $v_q(x,t)$ and $w_q (x,t)$ by setting \begin{align*}&v_q(x,t) = \sum_{r=1}^{q} e^{2\pi i (r^2t + rx)} \sum_{l=0}^{\lfloor N/q \rfloor} e^{2\pi i (l^2 q^2 t + lqx  )},\\  &w_{q}(x,t)  = \sum_{l=0}^{\lfloor N/q \rfloor} e^{2\pi i (l^2 q^2 t + lqx  )},  \ \ \ \ (x,t) \in \bigcup_{Q\in X_q} Q . \end{align*}

\noindent  We will now argue by induction on the scale. The main idea is that each $w_q$ can be viewed as a Weyl sum at scale $\sim N/q$, and so we can hope to use our inductive hypothesis to count the number of cubes in $X_q$. 

Note that since $|t- a/q| < 1/q^2$ for some $(a, q) = 1$ the classical Weyl bound tells us that $$|v_q(x,t)| \lesssim (\log(q)q)^{1/2} |w_q (x,t)|.$$ It now follows from the above estimates that  \begin{align} \label{eq:qContribution} \big(\sum_{Q \in X_q} \int_{Q} |u(x,t)|^4 dxdt \big)^{1/4} \lesssim & \big(\log(q)^2 q^2 \sum_{Q \in X_q} \int_{Q}|w_{q}(x, t)|^4 dx dt \big)^{1/4}  \\ \nonumber &+ \big(\int_{\bigcup_{Q \in X_q}}|E(x,t)|^4 dxdt\big)^{1/4} \end{align} and hence \begin{align} \label{eq:qContribution2} \sum_{q=N^{\delta}}^{c2^{m}\log(N)^2}\sum_{Q \in X_q} \int_{Q} |u(x,t)|^4 dxdt  \lesssim \sum_{q=N^{\delta}}^{c2^{m}\log(N)^2}&\log(q)^2 q^2 \sum_{Q \in X_q} \int_{Q}|w_{q}(x, t)|^4 dx dt  \\ \nonumber &+ O(N^{-4 + 2\alpha}2^{4m} \log(N)^8),  \end{align} with the second term on the right coming from the bound for the error $|E|$. Since $$2^{m} \log(N)^2 < N^{3/4}$$ the error term makes an acceptable contribution to \eqref{eq:oneDim} and we henceforth drop it from our discussion. 

Now observe that if we set $$w(x,t) = \sum_{l=0}^{\lfloor N/q \rfloor} e^{2\pi i(lx + l^2 t )  }$$ then we have \begin{equation} \label{eq:inductiveWeyl} w_q(x,t) = w(qx, q^2t).\end{equation} This identity puts us in a position to use induction at scale $N/q$, although we need to do a little more work to ensure that the induction hypotheses are satisfied. 

It is evident that $w_q$ is $(1/q)$-periodic in the $x$ variable and $(1/q^2)$-periodic in the $t$ variable. We exploit this fact with the following lemma. 
 
 \begin{lemma}\label{lem:qPeriod} Fix $q$ and define the collection $X_q$ as above. Then there exists a collection of rectangles $\overline{X}_q$ such that \begin{enumerate}
 	\item Each $\overline{Q} \in \overline{X}_q$ has dimensions approximately $N^{-2+\alpha} \times N^{-4 + 2\alpha}$
 	\item  $\overline{Q} \subset [0,1/q]\times [0, 1/q^2]$ for each $\overline{Q} \in \overline{X}_q$
 \item We have $$\sum_{Q \in X_q} \int_{Q} |w_{q}(x,t)|^4 dxdt \lesssim \sum_{\overline{Q} \in \overline{X}_q} \int_{\overline{Q}} |w_{q}(x,t)|^4 dxdt.$$ 
\item The collection $\overline{X}_{q}$ is a union of at most $O(q)$ one-dimensional collections of rectangles. \end{enumerate}  \end{lemma} 
 
 \begin{proof} The first three points follow from the fact that $w_{q}(x,t) = w(qx, q^2t)$ as described above, hence $w_{q}(x,t)$ is $(1/q)$-periodic in $x$ and $(1/q^2)$-periodic in $t$. Indeed we can let $\overline{Q}$ be the translate of $Q$ in $[0,1/q] \times [0,1/q^2]$ such that $$w_{q}(x + 1/q,t + 1/q^2)\mathbbm{1}_{Q}(x + n/q,t + m/q^2) = w_{q}(x,t)\mathbbm{1}_{\overline{Q}}(x,t), \ \ \ \ (x,t) \in [0, 1/q] \times [0,1/q^2]$$ for some integers $n,m$. The last point (4) follows from the fact that $\overline{Q}$ constructed in this matter can only overlap on the order of $O(q)$, since $X_q$ is one-dimensional.  
\end{proof}

 \noindent It is here that we are crucially using the one-dimensional hypothesis on the collection of rectangles, since this property ensures that there is limited overlap when shifting the rectangles as in Lemma \ref{lem:qPeriod}. We will see below that overlap on the order of $O(q)$ is admissible. 
 
 To make the last point more precise, define a map $Q \rightarrow \tau(Q)$ by letting $\tau(Q) = \overline{Q}$ be the copy of $Q$ mod $(1/q, 1/q^2)$ chosen in Lemma \ref{lem:qPeriod}. Recall from \eqref{eq:qContribution} that up to negligible error terms we have $$ \sum_{Q \in X_q} \int_{Q} |u(x,t)|^4 dxdt  \lesssim \log(q)^2 q^2 \sum_{Q \in X_q} \int_{Q}|w_{q}(x, t)|^4 dx dt.$$ By Lemma \ref{lem:qPeriod} we may therefore find a sub-collection $X'_{q} \subset X_{q}$ such that $\tau(X'_{q})$ is a one-dimensional collection contained in $[0,1/q]\times [0, 1/q^2]$ and such that \begin{equation} \label{eq:induct1}  \sum_{Q \in X_q} \int_{Q} |u(x,t)|^4 dxdt \lesssim \log(q)^2q^3 \sum_{Q \in X_q'} \int_{\tau(Q)}|w_{q}(x, t)|^4 dx dt. \end{equation} Since the $\tau(Q)$ are contained in $[0,1/q] \times [0,1/q^2]$ we can make a change of variables $(y,s) = (qx,q^2t)$ to obtain \begin{equation} \label{eq:induct2} \log(q)^2q^3 \sum_{Q \in X_q'} \int_{\tau(Q)}|w_{q}(x, t)|^4 dx dt \leq \log(q)^2 \sum_{R \in Y_q} \int_{R} |w(x,t)|^4 dx dt, \end{equation} where $Y_q$ is the image of $X_{q}'$ under the rescaling. In particular each $R \in Y_q$ is an axis-parallel rectangle of dimensions approximately $$ qN^{-2 + \alpha} \times q^2 N^{-4 + 2\alpha},$$ and each $R \in Y_q$ is contained in $[0,1]^2$. Moreover $Y_{q}$ is one-dimensional. Note that we have applied a version of the parabolic rescaling that plays a crucial role in the theory of the Euclidean operator $e^{it \Delta_{\R^n}}$. 
 
 Now using \eqref{eq:induct1}, \eqref{eq:induct2}, and the fact that $Y_q$ is one-dimensional, we conclude that \begin{equation}
\label{eq:induct3} \sum_{Q \in X_q} \int_{Q} |u(x,t)|^4 dxdt \lesssim N^{-4 + 2\alpha}\log(q)^2 q^2 \int_{\T} \sup_{0 < t < 1}|w(x,t)|^4 dx dt.
 \end{equation} Recall that $w(x,t) = \sum_{n=1}^{\lfloor N/q \rfloor} e^{2\pi i (nx + n^2 t)},$ so $w$ is a scale $ N/q$ Weyl sum. Then by our induction hypothesis $$\int_{\T}\sup_{0 < t < 1}|w(x,t)|^4 dx \leq C_{\epsilon}^4 N^{4 \epsilon} q^{-4\epsilon} N^{3} q^{-3}. $$ Inserting this into \eqref{eq:induct3} yields \begin{equation} \label{eq:induct4} N^{4-2\alpha} \sum_{Q \in X_q} \int_{Q} |u(x,t)|^4 dxdt \leq CC_{\epsilon}^4 q^{-4\epsilon}\log(q)^2 q^{-1} N^{3 + 4\epsilon}
 \end{equation} Recall $q \geq N^{\delta}$ and $\delta = \delta(\epsilon)$. We may therefore assume $N \gtrsim_{\epsilon} 1$ has been chosen large enough that $$Cq^{-2\epsilon}\log(q)^2 \leq 1.$$ Then by plugging \eqref{eq:induct4} into \eqref{eq:qContribution2} we arrive at \begin{align*} N^{4-2\alpha}\sum_{q=N^{\delta}}^{c2^{m}\log(N)^2}\sum_{Q \in X_q} \int_{Q} |u(x,t)|^4 dxdt  &\leq N^{-2\delta\epsilon}\sum_{q= N^{\delta}}^{c\log(N)^2 2^{m}} CC_{\epsilon}^4q^{-1} N^{3 + 4\epsilon} \\ &\leq N^{-2\delta \epsilon}C \log(N)^3 (C_{\epsilon}N^{3/4 + \epsilon})^{4}. \end{align*} Finally we may assume that $N \gtrsim_{\epsilon}1$ has been chosen large enough that $$ N^{-\delta \epsilon}C \log(N)^3 \leq (1/3)^4.$$  This shows that rectangles from $X_q$ with $q \geq N^{\delta}$ give an acceptable contribution to \eqref{eq:oneDim}. 
 
 This was the only remaining case, and so we conclude from \eqref{eq:oneDim} that \begin{align*}  \big(\int_{\T} \sup_{0 < t < 1}|u(x,t)|^4 dx \big)^{1/4} & \lesssim \log(N)^{1/4} \big(\int_{E_k} \sup_{0 < t < 1}|u(x,t)|^4 dx \big)^{1/4} \\  &\lesssim \log(N)^{1/4}  N^{1 -\alpha/2}\big(\sum_{Q_{\alpha}}  \int_{Q_{\alpha}} |u(x,t)|^4 dxdt \big)^{1/4} \\ & \leq C\log(N)^{1/4} N^{-\delta\epsilon}  C_{\epsilon} N^{3/4  +\epsilon} \\ &\leq C_{\epsilon} N^{3/4 + \epsilon}, \end{align*} the last line following as long as we choose $\delta < \epsilon^2$ such that $C N^{-\delta\epsilon} \log(N)^{1/4} \leq 1$. This completes the proof of Theorem \ref{prop:WeylMaximal}. 
 
 \begin{remark}
 	In this section we have worked with an equivalent version of the maximal problem that involves estimating $w_N$ on one-dimensional collections of axis-parallel rectangles in $\mathbb{T}^2$. We have chosen to present the argument in this manner to illustrate the connection with ideas of Du and Zhang \cite{DZ}, who prove the Euclidean analogue of Theorem \ref{prop:WeylMaximal} by analyzing the behavior of the associated Fourier extension operator on lower-dimensional unions of cubes in space-time. It is straightforward (and perhaps more natural) to rewrite the argument in this section to directly involve the maximal function. Indeed, the rectangles in $X_q$ can be replaced by intervals consisting of their projections to $[0,1]$, and the two-dimensional integrals used in this section can be replaced by one-dimensional integrals of the maximal function. The argument then goes through as above, with minor changes made. 
 	
 \end{remark}

\section{Corollaries of the maximal estimate}\label{sec:cor}

As a corollary of Theorem \ref{prop:WeylMaximal} we obtain the following discrete `one-dimensional' level set estimate for $w_{N}$. 

\begin{prop}\label{prop:mainWeylProp} Let $w_N$ denote a scale-$N$ quadratic Weyl sum as above. Let $A > 0$ be an arbitrary constant independent of other parameters. Fix $3/4 \leq \alpha < 1$. Let $X$ be a one-dimensional collection of rectangles $Q$ at scale $(N,\alpha)$ and suppose $$\|w_N\|_{L^{\infty}(Q)} \geq  AN^{\alpha}$$ for each $Q$. Then for every $\epsilon > 0$  there exists $C(\epsilon,A) > 0$ such that $$\#X \leq C(\epsilon,A)N^{\epsilon}N^{5(1-\alpha)}. $$ 
\end{prop}

\noindent Notice that the number of rectangles in a one-dimensional collection at scale $(N, \alpha)$ is always less than $CN^{2-\alpha}$, and so the proposition is trivial in the range $\alpha \leq 3/4$. 

\begin{proof} This is a consequence of the maximal estimate along with Lemma \ref{lem:const}. Indeed we have $$ \big( \sum_{Q \in X} \int_{Q} |w_N(x,t)|^4 dx dt \big)^{1/4} \gtrsim N^{\alpha } N^{-6/4 + 3\alpha/4} (\# X)^{1/4}, $$ and on the other hand since the collection is one-dimensional $$\big( \sum_{Q \in X} \int_{Q} |w_N(x,t)|^4 dx dt \big)^{1/4}  \lesssim N^{-1 + \alpha/2} \big(\int_{\T} \sup_{0<t<1}|w_N(x,t)|^4 dx \big)^{1/4} \leq C_{\epsilon}N^{\epsilon} N^{-1/4 + \alpha/2}.$$ Therefore $$(\# X)^{1/4} \lesssim C_{\epsilon}N^{2\epsilon} N^{5/4 - 5\alpha/4}$$ which implies the proposition.  
\end{proof}
Proposition \ref{prop:mainWeylProp} implies the following. 

\begin{cor}\label{cor:size} Let $\mathcal{S}_\alpha (N)$ denote the set of $x \in \T$ such that $$\sup_{0< t < 1}\bigg| \sum_{n=1}^{N} e^{2\pi i (xn + tn^2)} \bigg| \geq cN^{\alpha}.$$ Then $\mathcal{S}_{\alpha} (N)$ has measure $O_{\epsilon}(N^{3 - 4\alpha + \epsilon}).$  
\end{cor} \noindent The corollary follows by choosing intervals of the form $[a/N^{-2+\alpha}, (a+1)/N^{-2+\alpha}]$ which contain a point from $S_{\alpha}$ and then applying Proposition \ref{prop:mainWeylProp}.

As a consequence of Theorem \ref{prop:WeylMaximal} we also obtain the following refined Strichartz-type estimate, which should be compared to the main estimates of Du and Zhang in \cite{DZ}. 

\begin{cor}\label{cor:WeylCor} Let $w_N$ be a scale $N$ quadratic Weyl sum. Let $Q_j$ be a collection one-dimensional rectangles in $[0,1]^2$ at scale $(N,1)$ (so the dimensions are approximately $N^{-1 }\times N^{-2}$). Then for any $\epsilon > 0$ one has $$\big( \int_{\bigcup_j Q_j} |w_{N}(x,t)|^4 dx dt \big)^{1/4} \lesssim_{\epsilon} N^{1/4 + \epsilon}.$$ 
\end{cor}

\begin{proof} This is an immediate consequence of Theorem \ref{prop:WeylMaximal} since the collection is one-dimensional.  
\end{proof} \noindent Corollary \ref{cor:WeylCor} is sharp, as is seen by taking one of the $Q_j$ to contain $[0, 10^{-6}N^{-1}] \times [0, 10^{-6}N^{-2}]$. 

Finally, the following is also an easy consequence of Theorem \ref{prop:WeylMaximal}. 

\begin{cor} Let $q$ be an integer with $1 \leq q < N$. Suppose $\mathcal{P} = \{q, 2q, 3q,..., kq \}$ is contained in $[1,N]$, so that $\#\mathcal{P} \leq \lfloor N/q \rfloor$. Then for every $\epsilon > 0$ there is $C_{\epsilon} > 0$ such that $$\bigg\|\sup_{0 < t < 1} \big|\sum_{n\in \mathcal{P}} e^{2\pi i (n(\cdot) + n^2 t)}\big| \bigg\|_{L^{4}(\T)} \leq C_{\epsilon} q^{-3/4}N^{3/4 + \epsilon}.  $$
\end{cor}

\begin{proof} We have $$u(x,t) = \sum_{l=1}^{\lfloor N/q \rfloor } e^{2\pi i(lqx + l^2 q^2 t ) } = w(qx, q^2t),$$ where $$w(x,t) = \sum_{l=1}^{\lfloor N/q \rfloor} e^{2\pi i (lx + l^2 t)}.$$ In particular $u(x,t)$ is $(1/q)$-periodic in $x$ and $(1/q^2)$-periodic in $t$. This implies that $$\int_{0}^{1}\sup_{0<t<1} |u(x,t)|^4 dx = q\int_{0}^{1/q}\sup_{0<t<q^{-2}} |w(qx,q^2t)|^4 dx. $$ Making the change of variables $(y,s) = (qx, q^2t)$ then yields $$ \int_{0}^{1}\sup_{0<t<1} |u(x,t)|^4 dx = \int_{0}^{1}\sup_{0<s<1} |w(y,s)|^4 dy.$$ The claimed estimate now follows from Theorem \ref{prop:WeylMaximal}. 
\end{proof}

Note that if $u(x,t)$ is as in the previous corollary then $\|u(\cdot, 0)\|_{L^2} \sim (N/q)^{1/2}$. So the corollary says $$\|\sup_{0<t<1} |u(\cdot, t)\|_{L^4 (\T)} \leq C_{\epsilon}N^{\epsilon} q^{-1/4} N^{1/4+\epsilon} \|u(\cdot, 0)\|_{L^2 (\T)}, $$ which is an improvement over what can be expected in general (since the loss of $N^{1/4}$ is sharp for the complete Weyl sum, for example).

\section{The dimension of $\mathcal{O}_{\alpha}$}\label{sec:dimension}
We now prove Proposition \ref{prop:dim}. To prove the upper bound in the proposition we will apply our maximal estimate to a modification of the Weyl sum, following the completion method used by Chen and Shparlinski (see for example \cite{CS3}, Lemma 2.2). We will also use a fractal version of Theorem \ref{prop:WeylMaximal} which can be obtained by an argument of Eceizabarrena and Luc\`{a} \cite{EL}. To prove the lower bound we will appeal to a classical theorem of Jarn\'{\i}k.

 We begin by defining $$S_N(x,t) = \sum_{h=1}^{N} \frac{1}{h}\big|\sum_{n=1}^{N} e^{2\pi i hn/N} e^{2\pi i (xn + tn^2)} \big| = \sum_{h=1}^{N} \frac{1}{h}|w_N (x + h/N, t)|.$$  Then we have the following `completion' lemma proved in \cite{CS} and \cite{CS3}. 
\begin{lemma}[\cite{CS3}, Lemma 2.2]\label{lem:Slarge} For all $(x,t) \in \T^2$ and any $1 \leq N \leq M$ we have $$|w_{N}(x,t)| \lesssim |S_M (x,t)|.$$ 
\end{lemma} \noindent This lemma makes it easier to pass to dyadic scales, which is the main motivation for working with $S_N$.

\subsection{The upper bound} Recall that a Borel measure $\mu$ on $[0,1]$ is said to be $\beta$-dimensional with $0 < \beta \leq 1$ if $$c_{\beta}(\mu) := \sup_{I \subset \T} \frac{\mu(I)}{|I|^{\beta}} < \infty.$$ Here the supremum is taken over intervals in $[0,1]$. 

The upper bound in Proposition \ref{prop:dim} will follow after combining Theorem \ref{prop:WeylMaximal} with the following two lemmas. The proof will be similar to arguments used to study divergence sets for the Euclidean Schr\"{o}dinger operator in \cite{BBCR}, \cite{DZ}, \cite{LR}. We recall that $\mathcal{H}^{\gamma}$ denotes $\gamma$-dimensional Hausdorff measure. 

\begin{lemma}[Frostman's Lemma] \label{lem:Frostman} Suppose $\mathcal{S} \subset [0,1]$ is a Borel set and $\mathcal{H}^{\gamma}(S) > 0$ with $\gamma > 0$. Then there exists a nonzero $\gamma$-dimensional Borel measure $\mu_{\gamma}$ supported on $\mathcal{S}$. 

\end{lemma}  

\begin{lemma}[\cite{EL}]\label{lem:fractalMax} Fix $\epsilon > 0 $ and suppose $f \in H^{1/4 + \epsilon}(\T)$ is a function such that $$\big\| \sup_{0 < t < 1} |e^{it\Delta_{\T}} f| \|_{L^4 (\T)} \lesssim_{\epsilon} \|f\|_{H^{1/4 + \epsilon}}.$$ Then for any $\gamma$-dimensional measure $\mu_{\gamma}$ on $\T$ we also have $$\big\| \sup_{0 < t < 1} |e^{it\Delta_{\T}} f | \|_{L^4 (\T, d\mu_{\gamma})} \lesssim c_{\gamma}(\mu_{\gamma})^{1/4} \|f\|_{H^{s}}, \ \ \ \ \ s > \frac{1}{2} - \frac{\gamma}{4} $$
\end{lemma}

\noindent For a proof of Lemma \ref{lem:Frostman} see \cite{F} or \cite {M}. Lemma \ref{lem:fractalMax} is a special case of the proof of Proposition 5.2 in \cite{EL}.

We now fix $3/4 \leq \alpha \leq 1$ and suppose that  $\mathcal{H}^{\gamma}({O}_{\alpha}) > 0$ for $\gamma > 0$. Then by Lemma \ref{lem:Frostman} there exists a nonzero $\gamma$-dimensional Borel measure $\mu_{\gamma}$ supported on $\mathcal{O}_{\alpha}$. Note that if we define $$\mathcal{O}_{\alpha}(N) = \{x\in \T : \sup_{0 < t < 1}|w_{N}(x,t)| \geq N^{\alpha} \} $$ then we have \begin{equation}\label{eq:union2} \mathcal{O}_{\alpha} = \bigcap_{M \geq 1} \bigcup_{N \geq M} \mathcal{O}_{\alpha}(N).\end{equation} We fix an integer $M > 1$. For each $x \in \mathcal{O}_{\alpha}$ let $$N(x) = \min\{N\geq M : \sup_{0<t < 1}|w_{N}(x,t)| \geq N^{\alpha }\}.$$ Given $j \geq 1,$ define $$E_{j} = \{x\in \mathcal{O}_{\alpha} : N(x) \in [2^{j-1}M, 2^{j}M) \}.$$ The sets $E_j$ are pairwise disjoint by construction, and from \eqref{eq:union2} we have $\mathcal{O}_{\alpha} \subset \bigcup_{j\geq 1} E_j$ for each $M \geq 1$. Note that for each $x \in E_j$ we have $$1 \leq N(x)^{-\alpha}\sup_{t} |w_{N(x)}(x,t)| \lesssim 2^{-j\alpha}M^{-\alpha}\sup_{t} S_{2^j M}(x,t), $$ which follows from the definition of the $E_j$ and Lemma \ref{lem:Slarge}.  Therefore 
\begin{equation}\label{eq:measure1} \mu_{\gamma}(\T)^{1/4} = \big(\int_{\T}1 \ d\mu_{\gamma}(x) \big)^{1/4} \lesssim \big(\sum_{j=1}^{\infty}2^{-4\alpha j}M^{-4\alpha} \int_{E_j} \sup_{0< t < 1} |S_{2^{j}M}(x,t)|^4 \ d\mu_{\gamma}(x)\big)^{1/4}.
 \end{equation} Recall that if $\mu$ is a $\gamma$-dimensional measure and $\mu^{t}$ is a translate of $\mu$ then $\mu^{t}$ is also  $\gamma$-dimensional, with $c_{\gamma}(\mu^{t}) = c_{\gamma } (\mu)$. We therefore conclude from Theorem \ref{prop:WeylMaximal} and Lemma \ref{lem:fractalMax} that for any $\epsilon > 0$ we have \begin{align*} \big(\int_{E_j} \sup_{0 < t < 1} |S_{2^{j}M}(x,t)|^4 \ d\mu_{\gamma}(x)\big)^{1/4} &\leq \sum_{h=1}^{2^{j}M} \frac{1}{h} \big(\int_{E_j} \sup_{0 < t < 1} |w_{2^{j}M}(x + h/(2^{j}M),t)|^4 \ d\mu_{\gamma}(x)\big)^{1/4}   \\  &\leq C_{\epsilon}2^{j\epsilon}M^{\epsilon} 2^{j(1-\gamma/4)}M^{1-\gamma/4}. \end{align*} Inserting this into \eqref{eq:measure1} yields
 \begin{equation}\label{eq:measure2} \mu_{\gamma}(\T)^{1/4} \lesssim_{\epsilon} M^{1 -\alpha - \gamma/4 + \epsilon } \big(\sum_{j=1}^{\infty}2^{-j(4\alpha - 4 + \gamma - \epsilon )} \big)^{1/4}.
 \end{equation}  Let us now suppose that $\gamma > 4(1-\alpha)$, say $\gamma = 4(1-\alpha) + \delta$. Then from \eqref{eq:measure2} we get \begin{equation}\label{eq:measure3} \mu_{\gamma}(\T)^{1/4} \lesssim_{\epsilon} M^{-\delta/4 + \epsilon}\big(\sum_{j=1}^{\infty} 2^{-j(\delta - \epsilon )} )^{1/4} \lesssim_{\delta} M^{-\delta/8}
 \end{equation} if we choose $\epsilon = \delta/8$. But the implicit constant in \eqref{eq:measure3} is independent of $M$ and $M \geq 1$ was arbitrary, so we obtain a contradiction by letting $M \rightarrow \infty$. 
 
 Therefore $\mathcal{H}^{\gamma}(\mathcal{O}_{\alpha}) = 0$ for $\gamma > 4(1-\alpha).$ It follows that the dimension of $\mathcal{O}_{\alpha}$ is less than or equal to $4(1-\alpha)$, completing the proof.
 
 \subsection{The lower bound} The argument in this case is similar to the approach used by Chen and Shparlinski in \cite{CS2}, the main idea being that Weyl sums are known to be large at points which are well-approximated by certain rational numbers. There are also similar ideas used in \cite{OC} to estimate the dimension of sets related to the function $\sum_{n} \frac{1}{n} e^{2\pi i (xn + tn^2)},$ as mentioned in Section \ref{sec:otherWeyl}.  
 
  We will use the following theorem of Jarn\'{\i}k that gives the dimension of sets which are well-approximated by rationals. \begin{thm}[Jarn\'{\i}k] \label{thm:Jarnik} Suppose $\beta > 0$ and let $$\mathcal{J}_{\beta} = \{x\in \T : |x-a/q| \leq q^{-(2+\beta)} \text{ for infinitely many rationals } a/q \}.$$ Then the Hausdorff dimension of $\mathcal{J}_{\beta}$ is $\frac{2}{2+\beta}$. 
 \end{thm} \noindent See for example Theorem 10.3 in \cite{F2} or \cite{W} Section 9 for a proof. As observed in \cite{CS2}, the argument in \cite{F2} proves the the following stronger result.  
\begin{thm}\label{thm:JarnikOdd} Suppose $\beta > 0$ and let $$\mathcal{J}^{\text{odd} }_{\beta} = \{x\in \T : |x-a/q| \leq q^{-(2+\beta)} \text{ for infinitely many rationals } a/q, \ q \text{ odd} \}.$$ Then the Hausdorff dimension of $\mathcal{J}^{\text{odd}}_{\beta}$ is $\frac{2}{2+\beta}$. 
\end{thm} \noindent In fact we can further restrict to primes $q$. 

The following `major arc' estimate makes clear the role of the Jarn\'{\i}k theorem in our analysis. \begin{prop}\label{prop:WeylLower} Fix $ N > 1$ and supose $1 \leq q \leq N^{1/2}$ is an odd integer. Suppose $a/q, b/q$ are rationals with $1 \leq a,b < q$ and $(a,q) =  1$. Suppose $|x- b/q| \leq \frac{1}{100N}$ and $|t - a/q| \leq \frac{1}{100N^2}.$ Then $$\bigg| \sum_{n=1}^{N} e^{2\pi i (nx + n^2 t)} \bigg|\geq c \frac{N}{q^{1/2}}. $$
\end{prop} \noindent For a proof see \cite{O}, for example. 

 We fix a small parameter $\delta > 0$. Let's first suppose that $3/4 < \alpha < 1$. Then if $q$ is an integer chosen large enough (depending only on $\alpha$) we can find an integer $N_q$ such that $$100(N_q + 1) > q^{\frac{1}{2(1-\alpha)}} > 100 N_q  $$ and $$ q < N_{q}^{1/2}.$$ Now suppose that $|x-a/q|\leq q^{-\frac{1}{2(1-\alpha)}}.$ Then we also have $|x-a/q| < (100N_q)^{-1}$, and so if $q$ is odd Proposition \ref{prop:WeylLower} implies that $$\sup_{0<t<1}|w_{N_q}(x,t)| \geq cq^{-1/2}N_q \geq c' N_{q}^{-1 + \alpha}N_{q} \geq N_{q}^{\alpha - \delta}. $$
 \noindent We therefore conclude that \begin{equation}\label{eq:containment} \mathcal{J}^{\text{odd}}_{\beta} \subset \mathcal{O}_{\alpha - \delta}, \ \ \ \beta = \frac{4\alpha - 3 }{2(1-\alpha)}.
 \end{equation} Then from \eqref{eq:containment} and Theorem \ref{thm:JarnikOdd} it follows that $$\text{dim} (\mathcal{O}_{\alpha - \delta}) \geq \text{dim}(\mathcal{J}^{\text{odd}}_{\beta}) = \frac{2}{2+ \beta}, \ \ \ \ \beta = \frac{4\alpha - 3 }{2(1-\alpha)}.$$ Now direct calculation shows that $$\frac{2}{2+ \beta} = 4(1-\alpha)  $$ and so we conclude that $$\text{dim} (\mathcal{O}_{\alpha -\delta }) \geq 4(1-\alpha) $$ for arbitrarily small $\delta > 0$. It follows that if $3/4 < \alpha < 1$ and $\delta > 0$ is small enough we have $$\text{dim}(\mathcal{O}_{\alpha}) \geq 4(1-(\alpha + \delta)) = 4(1- \alpha) - 4\delta.$$ The desired lower bound follows by letting $\delta \rightarrow 0$. 
 
 It remains to consider the endpoint case $\alpha  = 3/4$. But this follows from the observation that $\mathcal{O}_{3/4} \supset \mathcal{O}_{3/4 + \epsilon}$ for any $\epsilon > 0$, along with the fact that $$\dim(\mathcal{O}_{3/4 + \epsilon}) = 1 - 4\epsilon $$ as proven above.

\Addresses
\end{document}